\documentclass[reqno, 12pt]{amsart}
\usepackage{amssymb}
\usepackage{amsfonts}
\usepackage{srcltx}
\usepackage{enumerate}
\usepackage{cite}

\newtheorem{theorem}{Theorem}[section]
\newtheorem{lemma}[theorem]{Lemma}

\newtheorem{corollary}[theorem]{Corollary}
\theoremstyle{definition}
\newtheorem{definition}[theorem]{Definition}

\newtheorem*{acknowledgement}{Acknowledgement}

\theoremstyle{remark}
\newtheorem{remark}[theorem]{Remark}

\DeclareMathOperator{\diam}{diam}

\begin{document}
\title[Nonexpansive-type mappings]{On the Suzuki nonexpansive-type mappings}
\author[A. Betiuk-Pilarska]{Anna Betiuk-Pilarska}
\author[A. Wi\'{s}nicki]{Andrzej Wi\'{s}nicki}

\begin{abstract}
It is shown that if $C$ is a nonempty convex and weakly compact subset of a
Banach space $X$ with $M(X)>1$ and $T:C\rightarrow C$ satisfies condition $%
(C)$ or is continuous and satisfies condition $(C_{\lambda })$ for some $%
\lambda \in (0,1),$ then $T$ has a fixed point. In particular, our theorem
holds for uniformly nonsquare Banach spaces. A similar statement is proved
for nearly uniformly noncreasy spaces.
\end{abstract}

\address{Anna Betiuk-Pilarska, Institute of Mathematics, Maria Curie-Sk\l %
odowska University, 20-031 Lublin, Poland}
\email{abetiuk@hektor.umcs.lublin.pl}
\address{Andrzej Wi\'{s}nicki, Institute of Mathematics, Maria Curie-Sk\l %
odowska University, 20-031 Lublin, Poland}
\email{awisnic@hektor.umcs.lublin.pl}
\maketitle

\section{Introduction}

Let $C$ be a nonempty subset of a Banach space $X$. A mapping $%
T:C\rightarrow X$ is said to be nonexpansive if
\begin{equation*}
\left\Vert Tx-Ty\right\Vert \leq \left\Vert x-y\right\Vert
\end{equation*}%
for $x,y\in C$. There is a large literature concerning fixed point theory of
nonexpansive mappings and their generalizations (see \cite{KiSi} and
references therein). Recently, Suzuki \cite{Su1} defined a class of
generalized nonexpansive mappings as follows.

\begin{definition}
A mapping $T:C\rightarrow X$ is said to satisfy condition $(C)$ if for all $%
x,y\in C,$%
\begin{equation*}
\frac{1}{2}\left\Vert x-Tx\right\Vert \leq \left\Vert x-y\right\Vert \text{\
\ implies \ }\left\Vert Tx-Ty\right\Vert \leq \left\Vert x-y\right\Vert .
\end{equation*}
\end{definition}

Subsequently the definition was widened in \cite{GLS}.

\begin{definition}
Let $\lambda \in (0,1).$ A mapping $T:C\rightarrow X$ is said to satisfy
condition $(C_{\lambda })$ if for all $x,y\in C,$%
\begin{equation*}
\lambda \left\Vert x-Tx\right\Vert \leq \left\Vert x-y\right\Vert \text{\ \
implies \ }\left\Vert Tx-Ty\right\Vert \leq \left\Vert x-y\right\Vert .
\end{equation*}
\end{definition}

It is not difficult to see that if $\lambda _{1}<\lambda _{2}$ then
condition $(C_{\lambda _{1}})$ implies condition $(C_{\lambda _{2}}).$
Several examples of mappings satisfying condition $(C_{\lambda })$ are given
in \cite{GLS, Su1}.

Two other related generalizations of a nonexpansive mapping have been
proposed in \cite{BDT} and \cite{LlMo}. Recall that a sequence $(x_{n})$ is
called an approximate fixed point sequence for $T$ (afps, for short) if $%
\lim_{n\rightarrow \infty }\left\Vert Tx_{n}-x_{n}\right\Vert =0.$

\begin{definition}[{see {\protect\cite[Def. 3.1]{BDT}}}]
A mapping $T:C\rightarrow X$ is said to satisfy condition $(\ast )$ if

\begin{enumerate}
\item[(i)] for each nonempty closed convex and $T$-invariant subset $D$ of $%
C $, $T$ has an afps in $D,$ and

\item[(ii)] For each pair of closed convex $T$-invariant subsets $D$ and $E$
of $C$, the asymptotic center $A(E,(x_{n}))$ of a sequence $(x_{n})$
relative to $E$ is $T$-invariant for each afps $(x_{n})$ in $D.$
\end{enumerate}
\end{definition}

\begin{definition}[{see {\protect\cite[Def. 3.1]{LlMo}}}]
A mapping $T:C\rightarrow X$ is said to satisfy condition $(L)$ if

\begin{enumerate}
\item[(i)] for each nonempty closed convex and $T$-invariant subset $D$ of $%
C $, $T$ has an afps in $D,$ and

\item[(ii)] For any afps $(x_{n})$ of $T$ in $C$ and for each $x\in C,$%
\begin{equation*}
\limsup_{n\rightarrow \infty }\left\Vert x_{n}-Tx\right\Vert \leq
\limsup_{n\rightarrow \infty }\left\Vert x_{n}-x\right\Vert .
\end{equation*}
\end{enumerate}
\end{definition}

It is easily seen that condition $(L)$ implies condition $(\ast ).$ One can
also prove that condition $(C)$ implies condition $(\ast )$ (see \cite[Lemma
6]{Su1}) and if $T:C\rightarrow C$ is continuous and satisfies condition $%
(C_{\lambda })$ for some $\lambda \in (0,1),$ then $T$ has a fixed point or
satisfies condition $(L)$ (see \cite[Theorem 4.7]{LlMo}). A natural question
arises whether a large collection of fixed point theorems for nonexpansive
mappings has its counterparts for mappings satisfying conditions $%
(C_{\lambda }),$ $(L)$ or $(\ast ).$ This is a non-trivial matter since some
constructions developed for nonexpansive mappings do not work properly in a
general case.

Let $C$ be a nonempty convex and weakly compact subset of a Banach space $X$%
. It was proved in \cite{Su1} that every mapping $T:C\rightarrow C$ which
satisfies condition $(C)$ has a fixed point when $X$ is UCED or satisfies
the Opial property, and in \cite{DIK}, when $X$ has property $(D).$ The
above results were generalized in \cite{LlMo} by showing that if $X$ has
normal structure, then every mapping $T:C\rightarrow C$ satisfying condition
$(L)$ has a fixed point. In particular, every continuous self-mapping of
type $(C_{\lambda })$ has a fixed point in this case. For a treatment of a
more general case of metric spaces and multivalued nonexpansive-type
mappings we refer the reader to \cite{EsLoNi} and the references given there.

Our paper is organized as follows. In Section 2 we prove that the mapping $%
T_{\gamma }=(1-\gamma )I+\gamma T,$ where $\gamma \in (0,1)$ is uniformly
asymptotically regular with respect to all $x\in C$ and all mappings from $C$
into $C$ which satisfy condition $(C_{\gamma }).$ We apply this result in
Section 3 to prove basic Lemmas \ref{3.3} and \ref{zn}. In Section 4 we are
able to adapt the proof of \cite[Theorem 9]{PrSz} and strenghten the result.
As a consequence, we show that if $C$ is a nonempty convex and weakly
compact subset of a nearly uniformly nonreasy space or a Banach space $X$
with $M(X)>1,$ then every mapping $T:C\rightarrow C$ which satisfies
condition $(C)$ and every continuous mapping $T:C\rightarrow C$ which
satisfies condition $(C_{\lambda })$ for some $\lambda \in (0,1)$ has a
fixed point. In particular, our theorems hold for both uniformly nonsquare
and uniformly noncreasy Banach spaces. In the case of uniformly nonsquare
spaces it answers Question 1 in \cite{DIK}.

\section{Asymptotic regularity}

Recall that a mapping $T:M\rightarrow M$ acting on a metric space $(M,d)$ is
said to be asymptotically regular if
\begin{equation*}
\lim_{n\rightarrow \infty }d(T^{n}x,T^{n+1}x)=0
\end{equation*}%
for all $x\in M.$ Ishikawa\ \cite{Is} proved that if $C$ is a bounded convex
subset of a Banach space $X$ and $T:C\rightarrow C$ is nonexpansive, then
the mapping $T_{\gamma }=(1-\gamma )I+\gamma T$ is asymptotically regular
for each $\gamma \in (0,1).$ Edelstein and O'Brien \cite{EdOb} showed that $%
T_{\gamma }$ is uniformly asymptotically regular over $x\in C,$ and Goebel
and Kirk \cite{GoKi1} proved that the convergence is uniform with respect to
all nonexpansive mappings from $C$ into $C$. The Ishikawa result was
extended in \cite[Lemma 6]{Su1} for mappings with condition $(C)$ and in
\cite[Theorem 4]{GLS} for mappings with condition $(C_{\lambda }).$ In this
section we prove the uniform version of that result. The proof follows in
part \cite[Lemma 1]{EdOb}.

\begin{theorem}
\label{2.1} Let $C$ be a bounded convex subset of a Banach space $X.$ Fix $%
\lambda \in (0,1),\gamma \in \lbrack \lambda ,1)$ and let $\mathcal{F}$
denote the collection of all mappings which satisfy condition $(C_{\lambda
}).$ Let $T_{\gamma }=(1-\gamma )I+\gamma T$ for $T\in \mathcal{F}.$ Then
for every $\varepsilon >0,$ there exists a positive integer $n_{0}$ such
that $\left\Vert T_{\gamma }^{n+1}x-T_{\gamma }^{n}x\right\Vert <\varepsilon
$ for every $n\geq n_{0},x\in C$ and $T\in \mathcal{F}$.
\end{theorem}

\begin{proof}
Without loss of generality we can assume that $\diam C=1.$ Suppose, contrary
to our claim, that there exists $\delta >0$ such that%
\begin{equation}
(\forall n_{0}>0)\ (\exists n\geq n_{0},x\in C,T\in \mathcal{F})\mathcal{\ }%
\Vert T_{\gamma }^{n+1}x-T_{\gamma }^{n}x\Vert \geq \delta .  \label{1}
\end{equation}%
Fix a positive integer $M>2/\delta $ and let $L=\lceil \frac{1}{\gamma
(1-\gamma )^{M}}\rceil $ denote the smallest integer not less than $\frac{1}{%
\gamma (1-\gamma )^{M}}.$ Then, by (\ref{1}), there exist $N>ML,x_{0}\in C$
and $T\in \mathcal{F}$ such that
\begin{equation*}
\Vert T_{\gamma }^{N+1}x_{0}-T_{\gamma }^{N}x_{0}\Vert \geq \delta .
\end{equation*}%
Let $x_{i}=T_{\gamma }^{i}x_{0}.$ Since%
\begin{equation*}
\lambda \Vert Tx_{i-1}-x_{i-1}\Vert =\frac{\lambda }{\gamma }\Vert T_{\gamma
}x_{i-1}-x_{i-1}\Vert \leq \Vert x_{i}-x_{i-1}\Vert ,
\end{equation*}%
$i=1,2,...,$ and $T$ satisfies condition $(C_{\lambda })$, we get
\begin{equation*}
\Vert Tx_{i}-Tx_{i-1}\Vert \leq \Vert x_{i}-x_{i-1}\Vert
\end{equation*}%
and hence%
\begin{equation*}
\Vert T_{\gamma }x_{i}-T_{\gamma }x_{i-1}\Vert \leq (1-\gamma )\Vert
x_{i}-x_{i-1}\Vert +\gamma \Vert Tx_{i}-Tx_{i-1}\Vert \leq \Vert
x_{i}-x_{i-1}\Vert
\end{equation*}%
for every positive integer $i.$ Thus
\begin{equation}
\Vert x_{1}-x_{0}\Vert \geq \Vert x_{2}-x_{1}\Vert \geq \ldots \geq \Vert
x_{N+1}-x_{N}\Vert \geq \delta  \label{2}
\end{equation}%
and%
\begin{equation}
\left\Vert \frac{1}{\gamma }(x_{i+1}-x_{i})-\frac{1-\gamma }{\gamma }%
(x_{i}-x_{i-1})\right\Vert =\Vert Tx_{i}-Tx_{i-1}\Vert \leq \Vert
x_{i}-x_{i-1}\Vert  \label{3}
\end{equation}%
for all $i=1,2,\ldots ,N$. We can now follow the arguments from \cite{EdOb}.
Notice that
\begin{equation*}
\lbrack \delta ,1]\subset \bigcup\nolimits_{i\mathbb{=}1}^{L}[b_{i},b_{i}+%
\gamma (1-\gamma )^{M}],
\end{equation*}%
where $b_{i}=\delta +(i-1)\gamma (1-\gamma )^{M}.$ Since $\left\{ \Vert
x_{Mi+1}-x_{Mi}\Vert :0\leq i\leq L\right\} $ has $L+1$ elements which
belong to $[\delta ,1]$ by $N>ML$ and (\ref{2}), it follows from the
pigeonhole principle that there exists an interval $I=[b,b+\gamma (1-\gamma
)^{M}]$ with $b\geq \delta $ and $0\leq i_{1}<i_{2}\leq L$ such that $\Vert
x_{Mi_{1}+1}-x_{Mi_{1}}\Vert ,\Vert x_{Mi_{2}+1}-x_{Mi_{2}}\Vert \in I.$
Hence by (\ref{2}),
\begin{equation}
\Vert x_{i+1}-x_{i}\Vert \in I\quad \text{for}\quad i=Mi_{1},Mi_{1}+1,\ldots
,Mi_{2}.  \label{4}
\end{equation}%
In particular, $\Vert x_{K+M+1}-x_{K+M}\Vert \in I,$ where $K=Mi_{1}.$
Select a functional $f\in S_{X^{\ast }}$ such that%
\begin{equation*}
f(x_{K+M+1}-x_{K+M})=\Vert x_{K+M+1}-x_{K+M}\Vert \geq b.
\end{equation*}

Then (\ref{3}) and (\ref{4}) imply
\begin{align*}
& \frac{1}{\gamma }f(x_{K+M+1}-x_{K+M})-\frac{1-\gamma }{\gamma }%
f(x_{K+M}-x_{K+M-1}) \\
& \leq \left\Vert \frac{1}{\gamma }(x_{K+M+1}-x_{K+M})-\frac{1-\gamma }{%
\gamma }(x_{K+M}-x_{K+M-1})\right\Vert \\
& \leq \Vert x_{K+M}-x_{K+M-1}\Vert \leq b+\gamma (1-\gamma )^{M},
\end{align*}%
so that
\begin{equation*}
\frac{b}{\gamma }-\frac{1-\gamma }{\gamma }f(x_{K+M}-x_{K+M-1})\leq b+\gamma
(1-\gamma )^{M}
\end{equation*}%
and hence
\begin{equation*}
f(x_{K+M}-x_{K+M-1})\geq b-\gamma ^{2}(1-\gamma )^{M-1}.
\end{equation*}%
Similarly,
\begin{align*}
b+(1-\gamma )^{M}\gamma & \geq \frac{1}{\gamma }f(x_{K+M}-x_{K+M-1})-\frac{%
1-\gamma }{\gamma }f(x_{K+M-1}-x_{K+M-2}) \\
& \geq \frac{1}{\gamma }\left( b-(1-\gamma )^{M}\gamma ^{2}\left( \frac{1}{%
1-\gamma }\right) \right) -\frac{1-\gamma }{\gamma }f(x_{K+M-1}-x_{K+M-2}),
\end{align*}%
and hence
\begin{equation*}
f(x_{K+M-1}-x_{K+M-2})\geq b-(1-\gamma )^{M}\gamma ^{2}\left( \frac{1}{%
1-\gamma }+\frac{1}{(1-\gamma )^{2}}\right) \geq b-\gamma (1-\gamma )^{M-2}.
\end{equation*}%
In general,
\begin{equation*}
f(x_{K+M+1-i}-x_{K+M-i})\geq b-\gamma (1-\gamma )^{M-i}
\end{equation*}%
for all $i=0,1,\ldots ,M$. Thus
\begin{align*}
f(x_{K+M+1})& \geq f(x_{K+M})+b \\
& \vdots \\
& \geq f(x_{K+M+1-i})+ib-\gamma ((1-\gamma )^{M-1}+...+(1-\gamma )^{M+1-i})
\\
& \vdots \\
& \geq f(x_{K+1})+Mb-\gamma ((1-\gamma )^{M-1}+...+(1-\gamma )) \\
& \geq f(x_{K+1})+Mb-1.
\end{align*}

But $b\geq \delta $ implies that $Mb\geq M\delta >2$, and so $\left\Vert
x_{K+M+1}-x_{K+1}\right\Vert \geq f(x_{K+M+1}-x_{K+1})>1$ contradicting the
assumption that $\diam C=1$.
\end{proof}

\section{Basic lemmas}

Let $C$ be a nonempty weakly compact convex subset of a Banach space $X$ and
$T:C\rightarrow C.$ It follows from the Kuratowski-Zorn lemma that there
exists a minimal (in the sense of inclusion) convex and weakly compact set $%
K\subset C$ which is invariant under $T.$ The first lemma below is a
counterpart of the Goebel-Karlovitz lemma (see \cite{Go, Ka}). It was proved
by Dhompongsa and Kaewcharoen \cite[Theorem 4.14]{DhKa} in the case of
mappings which satisfy condition $(C),$ and by Butsan, Dhompongsa and
Takahashi \cite[Lemma 3.2]{BDT} in the case of mappings satisfying condition
$(\ast ).$ Denote by
\begin{equation*}
r(K,(x_{n}))=\inf \{\limsup_{n\rightarrow \infty }\Vert x_{n}-x\Vert :x\in
K\}
\end{equation*}%
the asymptotic radius of a sequence $(x_{n})$ relative to $K$.

\begin{lemma}
\label{GKprim}Let $K$ be a nonempty convex weakly compact subset of a Banach
space $X$ which is minimal invariant under $T:K\rightarrow K.$ If $T$
satisfies condition $(\ast )$ (condition $(C)$, in particular), then there
exists an approximate fixed point sequence $(x_{n})$ for $T$ such that
\begin{equation*}
\lim_{n\rightarrow \infty }\Vert x_{n}-x\Vert =\inf \{r(K,(y_{n})):(y_{n})%
\text{ is an afps in }K\}\text{ }
\end{equation*}%
for every $x\in K.$
\end{lemma}

Llor\'{e}ns Fuster and Moreno G\'{a}lvez \cite[Th. 4.7]{LlMo} proved that if
$T:C\rightarrow C$ is continuous and satisfies condition $(C_{\lambda })$
for some $\lambda \in (0,1),$ then $T$ has a fixed point or satisfies
condition $(L).$ Since the set consisting of a single fixed point of $T$ is
minimal invariant under $T$ and condition $(L)$ implies condition $(\ast )$,
we obtain the following corollary.

\begin{lemma}
\label{GKbis}The conclusion of Lemma \ref{GKprim} is valid for continuous
mappings which satisfy condition $(C_{\lambda })$ for some $\lambda \in
(0,1).$
\end{lemma}

Now let $(x_{n})$ be a weakly null afps sequence for $T$ in $C.$ Fix $t<1$
and put $v_{n}=tx_{n}.$ The following technical lemma deals with the
behaviour of sequences $(T_{\gamma }^{k}v_{n})_{n\in \mathbb{N}},$ $%
k=1,2,....$

\begin{lemma}
\label{3.3} Assume that $T:C\rightarrow C$ satisfies condition $(C_{\lambda
})$ for some $\lambda \in (0,1).$ Fix $\gamma \in \lbrack \lambda ,1),$ a
positive integer $N,$ $0<\varepsilon <\frac{1}{10N}$ and $\frac{2}{3}%
+2N\varepsilon <t<1-2\varepsilon .$ Suppose that $(x_{n})$ is a weakly null
sequence in $C$ such that $\diam(x_{n})=1$ and the following conditions are
satisfied for every $n,m\in \mathbb{N}$ and $k=1,...,N$:

\begin{enumerate}
\item[(i)] a sequence $(T_{\gamma }^{k}v_{n})_{n\in \mathbb{N}},$ where $%
v_{n}=tx_{n},$ converges weakly to a point $y_{k}\in C$,

\item[(ii)] $\Vert T_{\gamma }^{k}v_{n}-T_{\gamma }^{k}v_{m}\Vert
>\liminf_{i}\Vert T_{\gamma }^{k}v_{n}-T_{\gamma }^{k}v_{i}\Vert
-\varepsilon ,$

\item[(iii)] $\min \{\Vert x_{n}\Vert ,\Vert x_{n}-x_{m}\Vert ,\Vert
x_{n}-y_{k}\Vert \}>1-\varepsilon ,$

\item[(iv)] $\Vert Tx_{n}-x_{n}\Vert <\varepsilon .$
\end{enumerate}

Then, for every $n,m\in \mathbb{N}$ and $k=1,...,N,$%
\begin{equation}
t-(k+2)\varepsilon <\left\Vert T_{\gamma }^{k}v_{n}-T_{\gamma
}^{k}v_{m}\right\Vert \leq t,  \label{con1}
\end{equation}%
\begin{equation}
1-t-\varepsilon <\left\Vert T_{\gamma }^{k}v_{n}-x_{n}\right\Vert
<1-t+k\varepsilon .  \label{con2}
\end{equation}
\end{lemma}

\begin{proof}
Fix $n,m\in \mathbb{N}$ and note that
\begin{equation*}
t-\varepsilon <\Vert v_{n}-v_{m}\Vert =t\Vert x_{n}-x_{m}\Vert \leq t,
\end{equation*}%
and%
\begin{equation*}
1-t-\varepsilon <\Vert x_{n}-v_{n}\Vert =(1-t)\left\Vert x_{n}\right\Vert
\leq (1-t)\diam(x_{n})\leq 1-t.
\end{equation*}%
Since
\begin{equation*}
\Vert Tx_{n}-x_{n}\Vert <\varepsilon <1-t-\varepsilon <\Vert
x_{n}-v_{n}\Vert ,\ (t<1-2\varepsilon ),
\end{equation*}%
it follows from condition $(C_{\lambda })$ that%
\begin{equation*}
\Vert Tx_{n}-Tv_{n}\Vert \leq \Vert x_{n}-v_{n}\Vert .
\end{equation*}%
Hence%
\begin{equation}
\Vert T_{\gamma }x_{n}-T_{\gamma }v_{n}\Vert \leq \gamma \Vert
Tx_{n}-Tv_{n}\Vert +(1-\gamma )\Vert x_{n}-v_{n}\Vert \leq \Vert
x_{n}-v_{n}\Vert \leq 1-t,  \label{fact1}
\end{equation}%
and%
\begin{align}
\Vert T_{\gamma }v_{n}-v_{n}\Vert & =\gamma \Vert Tv_{n}-v_{n}\Vert \leq
\Vert Tv_{n}-Tx_{n}\Vert +\Vert Tx_{n}-x_{n}\Vert +\Vert x_{n}-v_{n}\Vert
\label{fact2} \\
& <2\Vert x_{n}-v_{n}\Vert +\varepsilon \leq 2(1-t)+\varepsilon .  \notag
\end{align}%
We shall also use, for each $k\leq N,$ the following estimation which
follows from the weak lower semicontinuity of the norm:%
\begin{align}
1-\varepsilon & <\Vert x_{n}-y_{k}\Vert \leq \liminf_{m}\Vert
x_{n}-T_{\gamma }^{k}v_{m}\Vert  \label{fact3} \\
& \leq \Vert x_{n}-T_{\gamma }^{k}v_{n}\Vert +\liminf_{m}\Vert T_{\gamma
}^{k}v_{n}-T_{\gamma }^{k}v_{m}\Vert .  \notag
\end{align}%
Now we proceed by induction on $k.$

For $k=1,$ notice that
\begin{equation*}
\Vert T_{\gamma }v_{n}-v_{n}\Vert <2(1-t)+\varepsilon <t-\varepsilon <\Vert
v_{n}-v_{m}\Vert ,\ (t>\frac{2}{3}+\frac{2}{3}\varepsilon ),
\end{equation*}%
and it follows from condition $(C_{\lambda })$ that
\begin{equation}
\left\Vert T_{\gamma }v_{n}-T_{\gamma }v_{m}\right\Vert \leq \left\Vert
v_{n}-v_{m}\right\Vert \leq t.  \label{fact4}
\end{equation}%
Furthermore,%
\begin{equation}
\left\Vert T_{\gamma }v_{n}-x_{n}\right\Vert \leq \Vert T_{\gamma
}v_{n}-T_{\gamma }x_{n}\Vert +\Vert T_{\gamma }x_{n}-x_{n}\Vert
<1-t+\varepsilon ,  \label{fact5}
\end{equation}%
by (\ref{fact1}). To prove the reverse inequalities, notice that by (\ref%
{fact3}),%
\begin{equation*}
\left\Vert T_{\gamma }v_{n}-T_{\gamma }v_{m}\right\Vert >\liminf_{m}\Vert
T_{\gamma }v_{n}-T_{\gamma }v_{m}\Vert -\varepsilon >1-\varepsilon -\Vert
x_{n}-T_{\gamma }v_{n}\Vert -\varepsilon ,
\end{equation*}%
and it follows from (\ref{fact5}) that%
\begin{equation*}
\left\Vert T_{\gamma }v_{n}-T_{\gamma }v_{m}\right\Vert >1-\varepsilon
-(1-t+\varepsilon )-\varepsilon =t-3\varepsilon .
\end{equation*}%
Finally, by (\ref{fact3}) and (\ref{fact4}),%
\begin{equation*}
\left\Vert T_{\gamma }v_{n}-x_{n}\right\Vert >1-\varepsilon
-\liminf_{m}\Vert T_{\gamma }v_{n}-T_{\gamma }v_{m}\Vert \geq
1-t-\varepsilon .
\end{equation*}

Now suppose the lemma is true for a fixed $k<N.$ Then%
\begin{equation}
\left\Vert T_{\gamma }^{k+1}v_{n}-T_{\gamma }^{k+1}v_{m}\right\Vert \leq
\left\Vert T_{\gamma }^{k}v_{n}-T_{\gamma }^{k}v_{m}\right\Vert \leq t,
\label{fact6}
\end{equation}%
since (as in the proof of Theorem \ref{2.1})
\begin{align*}
\left\Vert T_{\gamma }T_{\gamma }^{k}v_{n}-T_{\gamma }^{k}v_{n}\right\Vert &
\leq \left\Vert T_{\gamma }^{k}v_{n}-T_{\gamma }^{k-1}v_{n}\right\Vert \leq
...\leq \left\Vert T_{\gamma }v_{n}-v_{n}\right\Vert \\
& <2(1-t)+\varepsilon <t-(k+2)\varepsilon <\left\Vert T_{\gamma
}^{k}v_{n}-T_{\gamma }^{k}v_{m}\right\Vert ,
\end{align*}%
(notice that $t>\frac{2}{3}+\frac{(k+3)\varepsilon }{3}$). Furthermore, by
induction assumption,%
\begin{equation*}
\Vert T_{\gamma }x_{n}-x_{n}\Vert <\varepsilon <1-t-\varepsilon <\Vert
x_{n}-T_{\gamma }^{k}v_{n}\Vert ,
\end{equation*}%
and hence
\begin{equation*}
\Vert T_{\gamma }^{k+1}v_{n}-T_{\gamma }x_{n}\Vert \leq \Vert T_{\gamma
}^{k}v_{n}-x_{n}\Vert .
\end{equation*}%
We thus get%
\begin{align}
\left\Vert T_{\gamma }^{k+1}v_{n}-x_{n}\right\Vert & \leq \Vert T_{\gamma
}^{k+1}v_{n}-T_{\gamma }x_{n}\Vert +\Vert T_{\gamma }x_{n}-x_{n}\Vert
\label{fact7} \\
& <\Vert T_{\gamma }^{k}v_{n}-x_{n}\Vert +\varepsilon <1-t+(k+1)\varepsilon .
\notag
\end{align}%
To prove the reverse inequalities, notice that by (ii), (\ref{fact3}) and (%
\ref{fact7}),%
\begin{align*}
\left\Vert T_{\gamma }^{k+1}v_{n}-T_{\gamma }^{k+1}v_{m}\right\Vert &
>\liminf_{i}\Vert T_{\gamma }^{k+1}v_{n}-T_{\gamma }^{k+1}v_{i}\Vert
-\varepsilon \\
& >1-\varepsilon -\Vert x_{n}-T_{\gamma }^{k+1}v_{n}\Vert -\varepsilon
>t-(k+3)\varepsilon .
\end{align*}%
Finally, by (\ref{fact3}) and (\ref{fact6}),%
\begin{equation*}
\left\Vert T_{\gamma }^{k+1}v_{n}-x_{n}\right\Vert >1-\varepsilon
-\liminf_{m}\Vert T_{\gamma }^{k+1}v_{n}-T_{\gamma }^{k+1}v_{m}\Vert \geq
1-t-\varepsilon ,
\end{equation*}%
and the proof is complete.
\end{proof}

We can now prove a counterpart of \cite[Lemma 2]{Do3} (see also \cite[%
Theorem 1]{JiLl}).

\begin{lemma}
\label{zn} Let $K$ be a convex weakly compact subset of a Banach space $X.$
Suppose that a mapping $T:K\rightarrow K$ satisfies condition $(C_{\lambda
}) $ for some $\lambda \in (0,1)$ and $(x_{n})$ is a weakly null,
approximate fixed point sequence for $T$ such that%
\begin{equation}
r=\lim_{n\rightarrow \infty }\Vert x_{n}-x\Vert =\inf \{r(K,(y_{n})):(y_{n})%
\text{ is an afps in }K\}>0\text{ }  \label{lim=r}
\end{equation}%
for every $x\in K.$ Then, for every $\varepsilon >0$ and $t\in (\frac{2}{3}%
,1),$ there exists a subsequence of $(x_{n}),$ denoted again $(x_{n}),$ and
a sequence $(z_{n})$ in $K$ such that

\begin{enumerate}
\item[(i)] $(z_{n})$ is weakly convergent,

\item[(ii)] $\Vert z_{n}\Vert >r(1-\varepsilon ),$

\item[(iii)] $\Vert z_{n}-z_{m}\Vert \leq rt,$

\item[(iv)] $\Vert z_{n}-x_{n}\Vert <r(1-t+\varepsilon )$\newline
for every $m,n\in \mathbb{N}.$
\end{enumerate}
\end{lemma}

\begin{proof}
Let us first notice that if $S:\frac{1}{\bar{r}}K\rightarrow \frac{1}{\bar{r}%
}K$ is defined by $Sy=\frac{1}{\bar{r}}T(\bar{r}y),$ then%
\begin{equation*}
\Vert Sy-y\Vert =\frac{1}{\bar{r}}\Vert T(\bar{r}y)-\bar{r}y\Vert
\end{equation*}%
and $S$ satisfies condition $(C_{\lambda }).$ It follows that a sequence $%
(x_{n})$ satisfies the assumptions of Lemma \ref{zn} if and only if a
sequence $(\frac{x_{n}}{r})$ satisfies these assumptions with $S$ and $\bar{r%
}=1$, i.e., $(\frac{x_{n}}{r})$ is a weakly null afps for $S:\frac{1}{r}%
K\rightarrow \frac{1}{r}K$ and
\begin{equation*}
1=\lim_{n\rightarrow \infty }\Vert \frac{x_{n}}{r}-y\Vert =\inf \{r(\frac{1}{%
r}K,(z_{n})):(z_{n})\text{ is an afps for }S\text{ in }\frac{1}{r}K\}
\end{equation*}%
for every $y\in \frac{1}{r}K.$

Therefore it suffices to prove the lemma for $r=1.$

We claim that for every $\varepsilon >0$ there exists $\delta (\varepsilon )$
such that if $x\in K$ and $\Vert Tx-x\Vert <\delta (\varepsilon )$ then $%
\Vert x\Vert >1-\varepsilon $. Indeed, otherwise, arguing as in \cite{Do3},
there exists $\varepsilon _{0}$ such that we can find $w_{n}\in K$ with $%
\Vert Tw_{n}-w_{n}\Vert <\frac{1}{n}$ and $\Vert w_{n}\Vert \leq
1-\varepsilon _{0}$ for every $n\in \mathbb{N}$. Then the sequence $(w_{n})$
is an approximate fixed point sequence in $K$, but $\limsup_{n\rightarrow
\infty }\Vert w_{n}\Vert \leq 1-\varepsilon _{0},$ which contradicts our
assumption that $\limsup_{n\rightarrow \infty }\Vert w_{n}\Vert \geq 1.$

Fix $\varepsilon >0,$ $t\in \left( \frac{2}{3},1\right) $ and $\gamma \in
\lbrack \lambda ,1).$ From Theorem \ref{2.1}, there exists $N>1$ such that
\begin{equation}
\Vert T_{\gamma }^{N+1}x-T_{\gamma }^{N}x\Vert <\gamma \delta (\varepsilon )
\label{8}
\end{equation}%
for every $x\in K$. Choose $\eta >0$ so small that $0<\eta <\min \left\{
\frac{1}{3(N+2)},\frac{\varepsilon }{N}\right\} $ and $\frac{2}{3}+N\eta
<t<1-2\eta $. Put $v_{n}=tx_{n}$ and consider sequences $(T_{\gamma
}^{k}v_{n})_{n\in \mathbb{N}}$ for $k=1,...,N$. We can assume, passing to
subsequences, that the double limits%
\begin{equation*}
\lim_{n,m\rightarrow \infty ,n\neq m}\Vert T_{\gamma }^{k}v_{n}-T_{\gamma
}^{k}v_{m}\Vert ,\ k=1,...,N,
\end{equation*}%
exist (see, e.g., \cite[Lemma 2.5]{SiSm}). Then, for sufficiently large $n,m$
$(n\neq m),$%
\begin{equation*}
\Vert T_{\gamma }^{k}v_{n}-T_{\gamma }^{k}v_{m}\Vert >\lim_{n,m\rightarrow
\infty ,n\neq m}\Vert T_{\gamma }^{k}v_{n}-T_{\gamma }^{k}v_{m}\Vert -\frac{%
\eta }{2}
\end{equation*}%
\begin{equation*}
=\limsup_{n\rightarrow \infty }\limsup_{m\rightarrow \infty }\Vert T_{\gamma
}^{k}v_{n}-T_{\gamma }^{k}v_{m}\Vert -\frac{\eta }{2}\geq
\liminf_{i\rightarrow \infty }\Vert T_{\gamma }^{k}v_{n}-T_{\gamma
}^{k}v_{i}\Vert -\eta ,
\end{equation*}%
$k=1,...,N.$ Therefore, applying (\ref{lim=r}) (with $r=1$) and passing to
subsequences again, we can assume that the assumptions $(i)-(iv)$ of Lemma %
\ref{3.3} are satisfied, i.e., $(x_{n})$ is weakly null, $\diam(x_{n})=1,$
and for every $n,m\in \mathbb{N}$ and $k=1,...,N,$

\begin{enumerate}
\item[(i)] $(T_{\gamma }^{k}v_{n})_{n\in \mathbb{N}}$ converges weakly to $%
y_{k}\in C,$

\item[(ii)] $\Vert T_{\gamma }^{k}v_{n}-T_{\gamma }^{k}v_{m}\Vert
>\liminf_{i}\Vert T_{\gamma }^{k}v_{n}-T_{\gamma }^{k}v_{i}\Vert -\eta $,

\item[(iii)] $\min \{\Vert x_{n}\Vert ,\Vert x_{n}-x_{m}\Vert ,\Vert
x_{n}-y_{k}\Vert \}>1-\eta $,

\item[(iv)] $\Vert Tx_{n}-x_{n}\Vert <\eta $.
\end{enumerate}

Denote $z_{n}=T_{\gamma }^{N}v_{n}$. It follows from Lemma \ref{3.3} that
for every $n,m\in \mathbb{N},$ we have
\begin{equation*}
\Vert z_{n}-z_{m}\Vert =\Vert T_{\gamma }^{N}v_{n}-T_{\gamma }^{N}v_{m}\Vert
\leq t,
\end{equation*}%
\begin{equation*}
\Vert z_{n}-x_{n}\Vert =\Vert T_{\gamma }^{N}v_{n}-x_{n}\Vert <1-t+N\eta
<1-t+\varepsilon
\end{equation*}%
and $(z_{n})$ is weakly convergent (to $y_{N}$).

Furthermore, by (\ref{8}),
\begin{equation*}
\Vert Tz_{n}-z_{n}\Vert =\frac{1}{\gamma }\Vert T_{\gamma
}^{N+1}v_{n}-T_{\gamma }^{N}v_{n}\Vert <\delta (\varepsilon )
\end{equation*}%
and consequently, $\Vert z_{n}\Vert >1-\varepsilon ,$ which completes the
proof.
\end{proof}

\section{Fixed point theorems}

Let $X$ be a Banach space without the Schur property. Recall \cite{PrSz}
that
\begin{align*}
d(\varepsilon ,x)& =\inf \{\limsup_{n\rightarrow \infty }\left\Vert
x+\varepsilon y_{n}\right\Vert -\left\Vert x\right\Vert :(y_{n})\text{ is
weakly null in }S_{X}\}, \\
b_{1}(\varepsilon ,x)& =\sup_{(y_{n})\in \mathcal{M}_{X}}\liminf_{n%
\rightarrow \infty }\left\Vert x+\varepsilon y_{n}\right\Vert -\left\Vert
x\right\Vert ,
\end{align*}%
where $\mathcal{M}_{X}$ denotes the set of all weakly null sequences $%
(y_{n}) $ in the unit ball $B_{X}$ such that%
\begin{equation*}
\limsup_{n\rightarrow \infty }\limsup_{m\rightarrow \infty }\left\Vert
y_{n}-y_{m}\right\Vert \leq 1.
\end{equation*}%
Applying tools from previous sections, we are led to the following
strengthening of Theorem 9 from \cite{PrSz}.

\begin{theorem}
\label{PrSzSu}Let $C$ be a nonempty convex weakly compact subset of a Banach
space $X$ without the Schur property. If there exists $\varepsilon \in (0,1)$
such that $b_{1}(1,x)<1-\varepsilon $ or $d(1,x)>\varepsilon $ for every $x$
in the unit sphere $S_{X},$ then every continuous mapping $T:C\rightarrow C$
which satisfies condition $(C_{\lambda })$ for some $\lambda \in (0,1),$ has
a fixed point. The assumption about the continuity of $T$ can be dropped if $%
T$ satisfies condition $(C).$
\end{theorem}

\begin{proof}
Assume that there exist a nonempty weakly compact convex set $C\subset X$
and a mapping $T:C\rightarrow C$ satisfying condition $(C)$ or, a continuous
mapping $T:C\rightarrow C$ satisfying condition $(C_{\lambda })$ for some $%
\lambda $, without a fixed point. Then, there exists a nonempty weakly
compact convex minimal and $T$-invariant subset $K\subset C$ with $\diam %
K>0. $ By Lemma \ref{GKprim} if $T$ satisfies condition $(C)$ or, by Lemma %
\ref{GKbis} in the other case, there exists an approximate fixed point
sequence $(x_{n})$ for $T$ in $K$ such that
\begin{equation*}
r=\lim_{n\rightarrow \infty }\Vert x_{n}-x\Vert =\inf \{r(K,(y_{n})):(y_{n})%
\text{ is an afps in }K\}>0\text{ }
\end{equation*}%
for every $x\in K.$ There is no loss of generality in assuming that $(x_{n})$
converges weakly to $0\in K$. Let $\varepsilon >0$ and $t=\frac{3}{4}.$
Lemma \ref{zn} yields a subsequence of $(x_{n})$, denoted again $(x_{n})$,
and a sequence $(z_{n})$ in $K$ such that

\begin{enumerate}
\item[(i)] $(z_{n})$ is weakly convergent to a point $z\in K$,
\end{enumerate}

and for every $n,m\in \mathbb{N}$

\begin{enumerate}
\item[(ii)] $\Vert z_{n}\Vert >r(1-\varepsilon )$,

\item[(iii)] $\Vert z_{n}-z_{m}\Vert \leq \frac{3}{4}r$,

\item[(iv)] $\Vert z_{n}-x_{n}\Vert <r(\frac{1}{4}+\varepsilon )$.
\end{enumerate}

Then%
\begin{equation*}
\liminf_{n\rightarrow \infty }\Vert z_{n}\Vert \geq r(1-\varepsilon ),
\end{equation*}%
\begin{equation*}
\limsup_{n\rightarrow \infty }\Vert z_{n}-z\Vert \leq \limsup_{n\rightarrow
\infty }\limsup_{m\rightarrow \infty }\Vert z_{n}-z_{m}\Vert \leq \frac{3}{4}%
r
\end{equation*}%
and%
\begin{equation}
r(\frac{1}{4}-\varepsilon )\leq \limsup_{n\rightarrow \infty }\Vert
z_{n}\Vert -\limsup_{n\rightarrow \infty }\Vert z_{n}-z\Vert \leq \Vert
z\Vert \leq \liminf_{n\rightarrow \infty }\Vert z_{n}-x_{n}\Vert \leq r(%
\frac{1}{4}+\varepsilon ).  \label{4_1}
\end{equation}%
Now we largely follow \cite[Theorem 9]{PrSz}. Let $u=\frac{z}{\Vert z\Vert }$
and $u_{n}=\frac{4}{3r}(z_{n}-z)$ for every $n$. Then $u\in S_{X}$, $(u_{n})$
is weakly null and
\begin{equation*}
\limsup_{n\rightarrow \infty }\limsup_{m\rightarrow \infty }\Vert
u_{n}-u_{m}\Vert =\frac{4}{3r}\limsup_{n\rightarrow \infty
}\limsup_{m\rightarrow \infty }\Vert z_{n}-z_{m}\Vert \leq 1.
\end{equation*}%
We may assume, passing to a subsequence, that $\lim_{n\rightarrow \infty
}\Vert u_{n}+u\Vert $ exists. Notice that
\begin{align*}
\Vert u_{n}+u\Vert & \geq \left\Vert \frac{4}{3r}(z_{n}-z)+\frac{4}{r}%
z\right\Vert -\left\Vert \frac{4}{r}z-\frac{z}{\Vert z\Vert }\right\Vert \\
& =\frac{4}{r}\left\Vert \frac{1}{3}z_{n}+\frac{2}{3}z\right\Vert
-\left\Vert \frac{4}{r}\left\Vert z\right\Vert -1\right\Vert ,
\end{align*}%
\begin{equation*}
\left\Vert \frac{1}{3}z_{n}+\frac{2}{3}z\right\Vert \geq \left\Vert
z_{n}\right\Vert -\frac{2}{3}\left\Vert z_{n}-z\right\Vert
\end{equation*}%
and%
\begin{equation*}
\left\Vert \frac{4}{r}\left\Vert z\right\Vert -1\right\Vert \leq
4\varepsilon .
\end{equation*}%
Hence%
\begin{equation*}
\lim_{n\rightarrow \infty }\Vert u_{n}+u\Vert \geq \frac{4}{r}\left(
r(1-\varepsilon )-\frac{2}{3}\frac{3}{4}r\right) -4\varepsilon
=2-8\varepsilon .
\end{equation*}%
It follows that $b_{1}(1,u)\geq 1-8\varepsilon .$

Now consider the weakly null sequence $y_{n}=\frac{4}{r}(z_{n}-z-x_{n}).$
Since
\begin{equation*}
\liminf_{n\rightarrow \infty }\Vert y_{n}\Vert \geq \frac{4}{r}%
(\lim_{n\rightarrow \infty }\Vert x_{n}\Vert -\limsup_{n\rightarrow \infty
}\Vert z_{n}-z\Vert )\geq 1,
\end{equation*}%
we have
\begin{align*}
\limsup_{n\rightarrow \infty }\Vert y_{n}+u\Vert & \leq
\limsup_{n\rightarrow \infty }\left\Vert y_{n}+\frac{4}{r}z\right\Vert
+\left\Vert \frac{z}{\Vert z\Vert }-\frac{4}{r}z\right\Vert \\
& \leq \frac{4}{r}r(\frac{1}{4}+\varepsilon )+4\varepsilon =1+8\varepsilon .
\end{align*}%
From \cite[Lemma 4]{PrSz} we conclude that also%
\begin{equation*}
\limsup_{n\rightarrow \infty }\Vert \frac{y_{n}}{\Vert y_{n}\Vert }+u\Vert
\leq \limsup_{n\rightarrow \infty }\Vert y_{n}+u\Vert \leq 1+8\varepsilon .
\end{equation*}%
Consequently, $d(1,u)\leq 8\varepsilon $ which contradicts our assumption.
\end{proof}

Theorem \ref{PrSzSu} is our main theorem which has several consequences. In
\cite{PrSz}, the notion of nearly uniformly nonreasy spaces (NUNC, for
short) was introduced. Recall that a Banach space $X$ is NUNC if it has the
Schur property or, for every $\varepsilon >0$ there is $t>0$ such that
\begin{equation*}
d(\varepsilon ,x)\geq t\ \ or\ \ b(t,x)\leq \varepsilon t\ \ \ \text{for
every }x\in S_{X},
\end{equation*}%
where%
\begin{equation*}
b(\varepsilon ,x)=\sup \{\liminf_{n\rightarrow \infty }\left\Vert
x+\varepsilon y_{n}\right\Vert -\left\Vert x\right\Vert :(y_{n})\text{ is
weakly null in }S_{X}\}.
\end{equation*}%
Corollary 7 in \cite{PrSz} shows that all uniformly noncreasy spaces,
introduced earlier by Prus, are NUNC.

\begin{theorem}
\label{NUNC}Let $C$ be a nonempty convex weakly compact subset of a nearly
uniformly noncreasy Banach space $X$. Then every continuous mapping $%
T:C\rightarrow C$ which satisfies condition $(C_{\lambda })$ for some $%
\lambda \in (0,1),$ has a fixed point. The assumption about the continuity
of $T$ can be dropped if $T$ satisfies condition $(C).$
\end{theorem}

\begin{proof}
If $X$ has the Schur property, then every weakly compact subset of $X$ is
compact in norm. Therefore every continuous mapping $T:C\rightarrow C$ which
satisfies condition $(C_{\lambda })$ for some $\lambda \in (0,1),$ has a
fixed point. Furthermore, if $T$ satisfies condition $(C),$ the continuity
assumption can be dropped by \cite[Theorem 2]{Su1} or \cite[Theorem 4]{Su1}.

If $X$ does not have the Schur property, we can argue as in the proof of
\cite[Corollary 11]{PrSz}.
\end{proof}

\begin{remark}
Notice that Example 6 in \cite{GLS} shows that the assumption about the
continuity of $T$ is necessary for $\lambda >\frac{3}{4}$. The situation is
unclear for $\lambda \in (\frac{1}{2},\frac{3}{4}].$
\end{remark}

Now we will study spaces with $M(X)>1.$ Recall that, for a given $a\geq 0,$%
\begin{equation*}
R(a,X)=\sup \{\liminf_{n\rightarrow \infty }\left\Vert y_{n}+x\right\Vert \},
\end{equation*}%
where the supremum is taken over all $x\in X$ with $\left\Vert x\right\Vert
\leq a$ and all weakly null sequences in the unit ball $B_{X}$ such that
\begin{equation*}
D[(y_{n})]=\limsup_{n\rightarrow \infty }\limsup_{m\rightarrow \infty
}\left\Vert y_{n}-y_{m}\right\Vert \leq 1.
\end{equation*}%
Notice that in our notation,%
\begin{equation}
R(a,X)=\sup_{\left\Vert x\right\Vert \leq a}(b_{1}(1,x)+\left\Vert
x\right\Vert ).  \label{4_3}
\end{equation}%
The modulus $R(\cdot ,X)$ was defined by Dom\'{\i}nguez Benavides in \cite%
{Do1} as a generalization of the coefficient $R(X)$ introduced by Garc\'{\i}%
a Falset \cite{Ga1}. He also defined the coefficient%
\begin{equation*}
M(X)=\sup \left\{ \frac{1+a}{R(a,X)}:a\geq 0\right\}
\end{equation*}%
and proved that the condition $M(X)>1$ implies that $X$ has the weak fixed
point property for nonexpansive mappings. We generalize this result to
mappings which satisfy condition $(C_{\lambda })$.

The following lemma is an analogue (with a minor correction) of \cite[%
Corollary 4.3 (a), (b), (c)]{GLM}.

\begin{lemma}
\label{4.1} Let $X$ be a Banach space. The following conditions are
equivalent:

\begin{enumerate}
\item[(a)] $M(X)>1,$\smallskip

\item[(b)] there exists $a>0$ such that $R(a,X)<1+a,$\smallskip

\item[(c)] for every $a>0,$ $R(a,X)<1+a.$
\end{enumerate}
\end{lemma}

\begin{proof}
First prove that $(a)\Rightarrow (b)$. Assume that $M(X)>1.$ Then there
exists $a\geq 0$ with $R(a,X)<1+a.$ If it occurs that $a=0$ then $R(b,X)\leq
R(0,X)+b<1+b$ for each $b\geq 0.$

The proof of $(b)\Rightarrow (c)$ follows the arguments from \cite{GLM}. We
will show that if $R(a,X)=1+a$ for some $a>0$, then $R(b,X)=1+b$ for all $%
b>0 $. Let us then suppose that $R(a,X)=1+a$ for some $a>0$ and consider
another number $b>0$. Fix $\eta \in (0,1).$ Since%
\begin{equation*}
R(a,X)=1+a>1+a-\eta \min \{1,a\},
\end{equation*}%
there exist $x\in X$ with $\Vert x\Vert \leq a$ and a weakly null sequence $%
(x_{n})$ in $B_{X}$ such that $\limsup_{n\rightarrow \infty
}\limsup_{m\rightarrow \infty }\Vert x_{n}-x_{m}\Vert \leq 1$ and
\begin{equation*}
\liminf_{n\rightarrow \infty }\Vert x_{n}+x\Vert >1+a-\eta \min \{1,a\}.
\end{equation*}%
For each $n\in \mathbb{N},$ choose a functional $f_{n}\in S_{X^{\ast }}$
with
\begin{equation*}
f_{n}(x_{n}+x)=\Vert x_{n}+x\Vert .
\end{equation*}%
We can assume, passing to a subsequence, that $\lim_{n\rightarrow \infty
}f_{n}(x_{n})$ exists. Since $B_{X^{\ast }}$ is $\mathit{{w}^{\ast }}$%
-compact, there exist a directed set $(\mathcal{A},\preceq )$ and a subnet $%
(f_{n_{\alpha }})_{\alpha \in \mathcal{A}}$ of $(f_{n})$ which is $\mathit{{w%
}^{\ast }}$-convergent to some $f\in B_{X^{\ast }}.$ Then%
\begin{equation*}
\lim_{\alpha }f_{n_{\alpha }}(x_{n_{\alpha }}+y)=\lim_{\alpha }f_{n_{\alpha
}}(x_{n_{\alpha }})+\lim_{\alpha }f_{n_{\alpha }}(y)=\lim_{\alpha
}f_{n_{\alpha }}(x_{n_{\alpha }})+f(y)
\end{equation*}%
for every $y\in X.$

For a fixed $\varepsilon >0$ find $n_{0}\in \mathbb{N}$ such that
\begin{equation*}
\Vert x_{n}+x\Vert >\liminf_{n\rightarrow \infty }\Vert x_{n}+x\Vert
-\varepsilon
\end{equation*}%
for every $n\geq n_{0}$. Then there exists $\alpha \in \mathcal{A}$ such
that $n_{\beta }\geq n_{0}$ for every $\beta \succeq \alpha $ and
consequently, since $\varepsilon >0$ is arbitrary,
\begin{equation*}
\liminf_{\alpha }\Vert x_{n_{\alpha }}+x\Vert =\sup_{\alpha \in \mathcal{A}%
}\inf_{\beta \succeq \alpha }\Vert x_{n_{\alpha }}+x\Vert \geq
\liminf_{n\rightarrow \infty }\Vert x_{n}+x\Vert .
\end{equation*}%
Thus%
\begin{align*}
1+a-\eta \min \{1,a\}& <\liminf_{n\rightarrow \infty }\Vert x_{n}+x\Vert
\leq \liminf_{\alpha }\Vert x_{n_{\alpha }}+x\Vert \\
& =\lim_{\alpha }f_{n_{\alpha }}(x_{n_{\alpha }}+x)=\lim_{\alpha
}f_{n_{\alpha }}(x_{n_{\alpha }})+f(x).
\end{align*}%
Since for each $n\geq 1,$
\begin{equation*}
f_{n}(x_{n})\leq \Vert x_{n}\Vert \leq 1
\end{equation*}%
and
\begin{equation*}
f(x)\leq \Vert x\Vert \leq a
\end{equation*}%
we get
\begin{equation*}
\lim_{\alpha }f_{n_{\alpha }}(x_{n_{\alpha }})>1-\eta \min \{1,a\}\geq 1-\eta
\end{equation*}%
and
\begin{equation*}
f(x)>a-\eta \min \{1,a\}\geq a(1-\eta ).
\end{equation*}%
Therefore,

\begin{align*}
\liminf_{n\rightarrow \infty }\Vert x_{n}+\frac{b}{a}x\Vert & \geq
\lim_{n\rightarrow \infty }f_{n}(x_{n}+\frac{b}{a}x)=\lim_{\alpha
}f_{n_{\alpha }}(x_{n_{\alpha }}+\frac{b}{a}x) \\
& =\lim_{\alpha }f_{n_{\alpha }}(x_{n_{\alpha }})+\frac{b}{a}f(x)>1-\eta
+b(1-\eta )=(1+b)(1-\eta ).
\end{align*}

Hence $R(b,X)\geq (1+b)(1-\eta )$ and, by the arbitrariness of $\eta >0,$ we
have $R(b,X)\geq 1+b$, which gives $(b)\Rightarrow (c).$

Clearly, $(c)\Rightarrow (a),$ and the lemma follows.
\end{proof}

Theorem \ref{PrSzSu} and Lemma \ref{4.1} give the following corollary.

\begin{theorem}
Let $C$ be a nonempty convex weakly compact subset of a Banach space $X$
with $M(X)>1.$ Then every mapping $T:C\rightarrow C$ which satisfies
condition $(C)$ and every continuous mapping $T:C\rightarrow C$ which
satisfies condition $(C_{\lambda })$ for some $\lambda \in (0,1),$ has a
fixed point.
\end{theorem}

\begin{proof}
If $X$ has the Schur property and $T:C\rightarrow C$ satisfies condition $%
(C),$ the continuity assumption can be dropped by \cite[Theorem 2]{Su1} as
in the proof of Theorem \ref{NUNC}.

Assume now that $X$ does not have the Schur property and set $\varepsilon
=2-R(1,X).$ Then, by Lemma \ref{4.1} (c), $\varepsilon \in (0,1).$ It
suffices to notice that from (\ref{4_3}),
\begin{equation*}
b_{1}(1,x)\leq R(1,X)-1=1-(2-R(1,X))
\end{equation*}%
for every $x\in S_{X}$, and apply Theorem \ref{PrSzSu}.
\end{proof}

Garc\'{\i}a Falset, Llor\'{e}ns Fuster and Mazcu\~{n}an Navarro \cite{GLM}
introduced another modulus, $RW(a,X),$ which plays an important role in
fixed point theory for nonexpansive mappings. Recall that, for a given $%
a\geq 0,$%
\begin{equation*}
RW(a,X)=\sup \min \{\liminf_{n}\left\Vert x_{n}+x\right\Vert
,\liminf_{n}\left\Vert x_{n}-x\right\Vert \},
\end{equation*}%
where the supremum is taken over all $x\in X$ with $\left\Vert x\right\Vert
\leq a$ and all weakly null sequences in the unit ball $B_{X},$ and,
\begin{equation*}
MW(X)=\sup \left\{ \frac{1+a}{RW(a,X)}:a\geq 0\right\} .
\end{equation*}%
It was proved in \cite[Theorem 3.3]{GLM} that if $B_{X^{\ast }}$ is $w^{\ast
}$-sequentially compact, then $M(X)\geq MW(X).$ Since $B_{X^{\ast }}$ is $%
w^{\ast }$-sequentially compact if $X$ is separable, we obtain the following
corollary.

\begin{corollary}
Let $C$ be a nonempty convex weakly compact subset of Banach space $X$ with $%
MW(X)>1.$ Then every mapping $T:C\rightarrow C$ which satisfies condition $%
(C)$ and every continuous mapping $T:C\rightarrow C$ which satisfies
condition $(C_{\lambda })$ for some $\lambda \in (0,1),$ has a fixed point.
\end{corollary}

Recall that a Banach space $X$ is uniformly nonsquare if
\begin{equation*}
J(X)=\sup_{x,y\in S_{X}}\min \left\{ \left\Vert x+y\right\Vert ,\left\Vert
x-y\right\Vert \right\} <2.
\end{equation*}%
In \cite{GLM}, a characterization of reflexive Banach spaces with $MW(X)>1$
is given. In particular (see \cite[Corollary 5.1]{GLM}), all uniformly
nonsquare Banach spaces fulfill this condition. Thus we obtain the following
corollary which answers Question 1 in \cite{DIK}.

\begin{corollary}
Let $C$ be a nonempty convex weakly compact subset of a uniformly nonsquare
Banach space. Then every mapping $T:C\rightarrow C$ which satisfies
condition $(C)$ and every continuous mapping $T:C\rightarrow C$ which
satisfies condition $(C_{\lambda })$ for some $\lambda \in (0,1),$ has a
fixed point.
\end{corollary}

\begin{remark}
It is not known whether our results are valid for mappings satisfying
property $(L)$ or $(\ast ).$
\end{remark}

\begin{acknowledgement}
The authors thank Mariusz Szczepanik for helpful discussions and drawing
their attention to Theorem 9 in \cite{PrSz}.
\end{acknowledgement}

\bigskip


\begin{thebibliography}{99}
\bibitem{BDT} T. Butsan, S. Dhompongsa, W. Takahashi, A fixed point theorem
for pointwise eventually nonexpansive mappings in nearly uniformly convex
Banach spaces, Nonlinear Anal. 74 (2011), 1694--1701.

\bibitem{DhKa} S. Dhompongsa, A. Kaewcharoen, Fixed point theorems for
nonexpansive mappings and Suzuki-generalized nonexpansive mappings on a
Banach lattice, Nonlinear Anal. 71 (2009), 5344--5353.

\bibitem{DIK} S. Dhompongsa, W. Inthakon, A. Kaewkhao, Edelstein's method
and fixed point theorems for some generalized nonexpansive mappings, J.
Math. Anal. Appl. 350 (2009), 12--17.

\bibitem{Do1} T. Dom\'{\i}nguez Benavides, A geometrical coefficient
implying the fixed point property and stability results, Houston J. Math. 22
(1996), 835--849.

\bibitem{Do3} T. Dom\'{\i}nguez Benavides, A renorming of some nonseparable
Banach spaces with the fixed point property, J. Math. Anal. Appl. 350
(2009), 525--530.

\bibitem{EdOb} M. Edelstein, R. C. O'Brien, Nonexpansive mappings,
asymptotic regularity and successive approximations, J. London Math. Soc.
(2) 17 (1978), 547--554.

\bibitem{EsLoNi} R. Esp\'{\i}nola, P. Lorenzo, A. Nicolae, Fixed points,
selections and common fixed points for nonexpansive-type mappings, J. Math.
Anal. Appl. 382 (2011), 503--515.

\bibitem{Ga1} J. Garc\'{\i}a Falset, Stability and fixed points for
nonexpansive mappings, Houston J. Math. 20 (1994), 495--506.

\bibitem{GLM} J. Garc\'{\i}a Falset, E. Llor\'{e}ns Fuster, E. M. Mazcu\~{n}%
an Navarro, Uniformly nonsquare Banach spaces have the fixed point property
for nonexpansive mappings, J. Funct. Anal. 233 (2006), 494--514.

\bibitem{GLS} J. Garc\'{\i}a Falset, E. Llor\'{e}ns Fuster, T. Suzuki, Fixed
point theory for a class of generalized nonexpansive mappings, J. Math.
Anal. Appl. 375 (2011), 185--195.

\bibitem{Go} K. Goebel, On the structure of minimal invariant sets for
nonexpansive mappings, Ann. Univ. Mariae Curie-Sk{\l }odowska Sect. A 29
(1975), 73--77.

\bibitem{GoKi1} K. Goebel, W. A. Kirk, Iteration processes for nonexpansive
mappings, in: Topological Methods in Nonlinear Functional Analysis, S. P.
Singh, S. Thomeier, B. Watson (eds.), AMS, Providence, R.I., 1983, 115--123.

\bibitem{KiSi} Handbook of Metric Fixed Point Theory, W. A. Kirk, B. Sims
(eds.), Kluwer Academic Publishers, Dordrecht, 2001.

\bibitem{Is} S. Ishikawa, Fixed points and iteration of a nonexpansive
mapping in a Banach space. Proc. Amer. Math. Soc. 59 (1976), no. 1, 65--71.

\bibitem{JiLl} A. Jim\'{e}nez-Melado, E. Llor\'{e}ns Fuster, Opial modulus
and stability of the fixed point property, Nonlinear Anal. 39 (2000),
341--349.

\bibitem{Ka} L. A. Karlovitz, Existence of fixed points of nonexpansive
mappings in a space without normal structure, Pacific J. Math. 66 (1976),
153--159.

\bibitem{LlMo} E. Llor\'{e}ns Fuster, E. Moreno G\'{a}lvez, The fixed point
theory for some generalized nonexpansive mappings, Abstr. Appl. Anal. 2011,
Art. ID 435686, 15 pp.

\bibitem{PrSz} S. Prus, M. Szczepanik, Nearly uniformly noncreasy Banach
spaces, J. Math. Anal. Appl. 307 (2005), 255--273.

\bibitem{SiSm} B. Sims, M. A. Smyth, On some Banach space properties
sufficient for weak normal structure and their permanence properties, Trans.
Amer. Math. Soc. 351 (1999), 497--513.

\bibitem{Su1} T. Suzuki, Fixed point theorems and convergence theorems for
some generalized nonexpansive mappings, J. Math. Anal. Appl. 340 (2008),
1088--1095.
\end{thebibliography}
\end{document}